\documentclass[a4paper]{article}

\usepackage{graphicx}
\usepackage{color}
\usepackage{stmaryrd}
\usepackage[all]{xy}

\usepackage[width=128mm,height=210mm,centering]{geometry}

\usepackage[numbers]{natbib}
\usepackage{natbibspacing}
\renewcommand{\bibfont}{\small}

\definecolor{myurlcolor}{rgb}{0.1,0.1,0.8}
\definecolor{mylinkcolor}{rgb}{0.05,0.05,0.4}
\usepackage{hyperref}
\hypersetup{colorlinks,
linkcolor=mylinkcolor,
citecolor=mylinkcolor,
urlcolor=mylinkcolor}

\usepackage{tocloft}
\usepackage{latexsym}
\usepackage{amssymb,amsmath}
\usepackage{color}

\newcommand{\such}{:}
\newcommand{\demph}[1]{\textbf{\textup{#1}}}
\newcommand{\iso}{\cong}
\newcommand{\of}{\circ}
\newcommand{\sub}{\subseteq}
\newcommand{\cell}[4]{\put(#1,#2){\makebox(0,0)[#3]{#4}}}
\DeclareMathOperator{\im}{im}
\newcommand{\from}{\colon}
\newcommand{\vc}[1]{\mathbf{#1}} % Vector
\newcommand{\bhardref}[1]{(#1)}
\newcommand{\lengths}{\setlength{\unitlength}{1mm}%
\setlength{\fboxsep}{0pt}}
\DeclareMathOperator{\Nil}{Nil}
\DeclareMathOperator{\Lin}{Lin}
\DeclareMathOperator{\Aut}{Aut}
\DeclareMathOperator{\spn}{span}

\usepackage[amsmath,thmmarks]{ntheorem}

\newtheorem{thm}{Theorem}%[section]

\newtheorem{lemma}[thm]{Lemma}

\theorembodyfont{\normalfont}

\theoremstyle{nonumberplain}
\theoremsymbol{\ensuremath{\square}}
\qedsymbol{\ensuremath{\square}}

\newtheorem{proof}{Proof}

\newcommand{\theoremtobeproved}{}
\newtheorem{pfoftheorem}{Proof of \theoremtobeproved}

\title{\vspace{-10mm}%
The probability that an operator is nilpotent}
\author{Tom Leinster%
\thanks{School of Mathematics, University of Edinburgh, Scotland; 
  Tom.Leinster@ed.ac.uk}}
\date{\vspace{-5ex}}

\begin{document}

\sloppy
\maketitle

\begin{abstract}
\noindent
Choose a random linear operator on a vector space of finite cardinality
$N$: then the probability that it is nilpotent is $1/N$.  This is a linear
analogue of the fact that for a random self-map of a set of cardinality
$N$, the probability that some iterate is constant is $1/N$.  The first
result is due to Fine, Herstein and Hall, and the second is essentially
Cayley's tree formula.  We give a new proof of the result on nilpotents,
analogous to Joyal's beautiful proof of Cayley's formula.  It uses only
general linear algebra and avoids calculation entirely.
\end{abstract}

\section{Introduction}

This note is about the theorem that on a finite-dimensional vector space
$X$ over a finite field, the probability of a randomly-chosen linear
operator being nilpotent is $1$ over the number of elements of $X$.  If the
field has order $q$ and $\dim X = n$ then there are $q^n$ elements of $X$
and $q^{n^2}$ operators on $X$, so an equivalent result is that $X$ admits
$q^{n(n - 1)}$ nilpotents.

This theorem was first published in 1958 by Fine and Herstein~\cite{FiHe},
who gave a calculation-heavy proof using the theory of partitions.
Subsequently, others found proofs needing less calculation: first
Gerstenhaber~\cite{Gers} in 1961 (avoiding partitions), then
Kaplansky~\cite{Kapl} in 1990 (with an inclusion-exclusion argument), then
Crabb~\cite{Crab} in 2006 (imitating the proof of Cayley's formula by
Pr\"ufer codes), then Brouwer, Gow and Sheekey~\cite{BGS} in 2014 (with a
very efficient proof using the Fitting decomposition and $q$-binomial
coefficients). Brouwer, Gow and Sheekey also mention unpublished 1955 lecture
notes of Philip Hall, stating that he gave two proofs, `one involving a
form of M\"obius inversion, the other exploiting the theory of partitions'
(\cite{BGS}, Section~2.1).

This note gives a new proof requiring no calculation or manipulation of
algebraic expressions at all---at least, granted some standard linear
algebra.  It is analogous to Joyal's beautiful proof of Cayley's formula
(\cite{JoyaTCS}, p.~16), in the following sense.

Cayley's formula states that for a finite set $X$ with $N > 0$ elements,
the number of (unrooted) trees with vertex-set $X$ is $N^{N - 2}$.  A
\demph{rooted} tree is a tree together with a choice of vertex, called the
root; there are $N^{N - 1}$ of these.  They can be identified with the
functions $T \from X \to X$ that are \demph{eventually constant}, meaning
that $T^k = T \of \cdots \of T$ is constant for some $k \geq 0$.  Indeed,
if we orient the edges of a rooted tree towards the root, the resulting
diagram depicts an eventually constant function, with the root $z$ as the
eventual constant value and an edge from $x$ to $T(x)$ for each $x \neq z$.
Hence there are $N^{N - 1}$ eventually constant functions $X \to
X$.  Equivalently, the probability of a random function $X \to X$ being
eventually constant is $N^{N - 1}/N^N = 1/N$.

When $X$ is a vector space, the eventually constant linear maps $T \from X
\to X$ are exactly the \demph{nilpotent} operators (those satisfying $T^k =
0$ for some $k$).  So Cayley's formula and the theorem on nilpotents both
state, in different contexts, the probability that a random self-map is
eventually constant.

Joyal proved Cayley's formula by constructing a
bijection  
\begin{equation}
\label{eq:cayley-joyal}
\{\text{trees on } X \} \times X \times X
\iso
\{\text{functions } X \to X \}
\end{equation}
for any nonempty finite set $X$. His construction depends on an arbitrary
choice, for each $V \sub X$, of a bijection between the total orders on $V$
and the permutations of $V$.  It runs as follows.  An element of the
left-hand side of~\eqref{eq:cayley-joyal} is a tree with chosen vertices
$v$ and $v'$. There is a unique path from $v$ to $v'$, and the set $V$ of
vertices along that path is naturally ordered
(Fig.~\ref{fig:cayley}\bhardref{a}).
\begin{figure}
\centering%
\lengths%
% \fbox{%
\begin{picture}(120,52)(0,-2)
\cell{0}{3}{bl}{\includegraphics[width=28\unitlength]{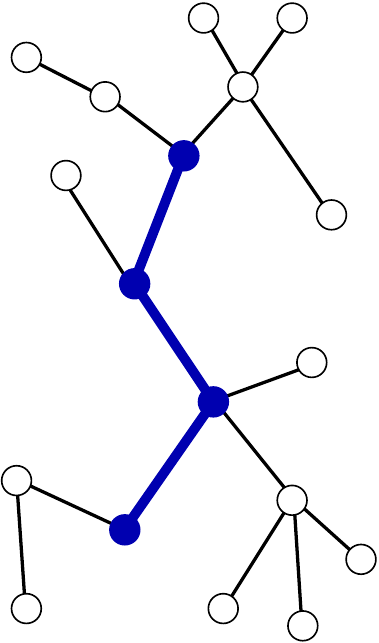}}
\cell{14}{-1}{c}{(a)}
\cell{10}{23}{c}{$V$}
\cell{11}{38.5}{c}{$v$}
\cell{12.5}{11}{c}{$v'$}
\cell{46}{3}{bl}{\includegraphics[width=28\unitlength]{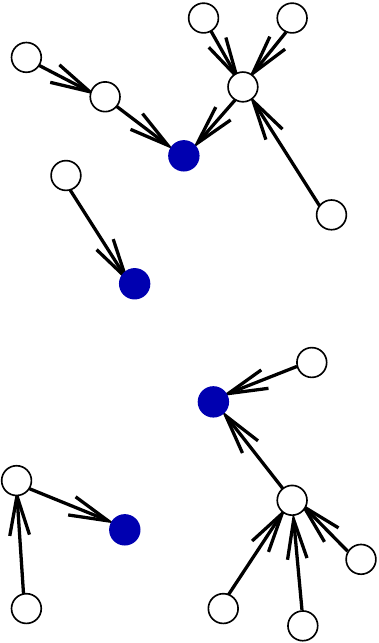}}
\cell{60}{-1}{c}{(b)}
\cell{57}{38}{c}{$0$}
\cell{53.5}{29}{c}{$1$}
\cell{59}{20.5}{c}{$2$}
\cell{58}{11.5}{c}{$3$}
\cell{92}{3}{bl}{\includegraphics[width=28\unitlength]{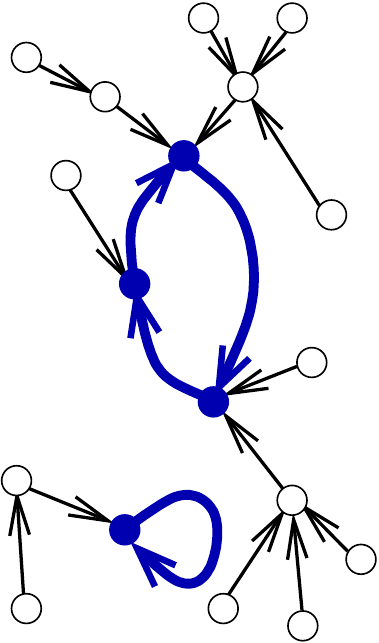}}
\cell{106}{-1}{c}{(c)}
\end{picture}%
% }
\caption{Joyal's proof of Cayley's formula}
\label{fig:cayley}  
\end{figure}
Thus, a tree on $X$ with two distinguished vertices amounts to a totally
ordered subset $V$ of $X$ together with a family of rooted trees
partitioning $X$, the roots being the elements of $V$
(Fig.~\ref{fig:cayley}\bhardref{b}).  As before, we can harmlessly
orient the edges of these trees towards their roots.

We can equivalently replace the total order on $V$ by the corresponding
permutation.  This produces a diagram as in Fig.~\ref{fig:cayley}\bhardref{c}
(where the permutation shown is necessarily arbitrary).  But such a diagram
is simply the graph of a function $X \to X$, with $V$ as its set of
periodic points.  This completes the proof.

Our proof of the formula for nilpotents follows a similar pattern. For
example, Joyal's argument uses the decomposition of an endomorphism of a
finite set into a permutation and some eventually constant functions, and
our argument uses the decomposition of an endomorphism of a
finite-dimensional vector space into an automorphism and a nilpotent.
However, the translation to the linear context is not entirely mechanical.
An important difference is that whereas a subset of a set has only one
complement, a subspace of a vector space has many complementary subspaces.

One feature of Joyal's proof is that the total orders and permutations of a
finite set $V$ are in \emph{non-canonical} bijection.  This situation can
be understood through Joyal's theory of species~\cite{JoyaTCS}, or by
noting that we have a \demph{torsor}: a nonempty set $S$ acted on by a
group $G$ in such a way that for each $s, s' \in S$, there is a unique $g
\in G$ satisfying $gs = s'$.  Every element $s_0 \in S$ defines a bijection
$G \iso S$ by $g \leftrightarrow gs_0$, but in general there is no
\emph{canonical} bijection $G \iso S$: for which element of $S$ would
correspond to the identity element of $G$?  In the case at hand, $G$ is the
group of permutations of $V$ acting on the set of total orders on $V$. We
will use an analogous torsor in the linear context.

\section{Background linear algebra}
\label{sec:la}

Throughout, all vector spaces are over an arbitrary field, not necessarily
finite.  

Subspaces $U$ and $V$ of a vector space $X$ are
\demph{complementary} if $U \cap V = \{0\}$ and $U + V = X$; we also say
that $U$ is a \demph{complement} of $V$.  Although a subspace generally has
many complements, any two are canonically isomorphic.  Indeed, let $U$ and
$W$ be complements of $V \sub X$; then there is an isomorphism $i \from U
\to W$ defined by taking $i(u)$ to be the unique element of $W$ such that
$i(u) - u \in V$ (Fig.~\ref{fig:comps}).

\begin{figure}
\centering%
\lengths%
\setlength{\unitlength}{1.1mm}
\begin{picture}(40,20)
\cell{0}{0}{bl}{\includegraphics[width=40\unitlength]{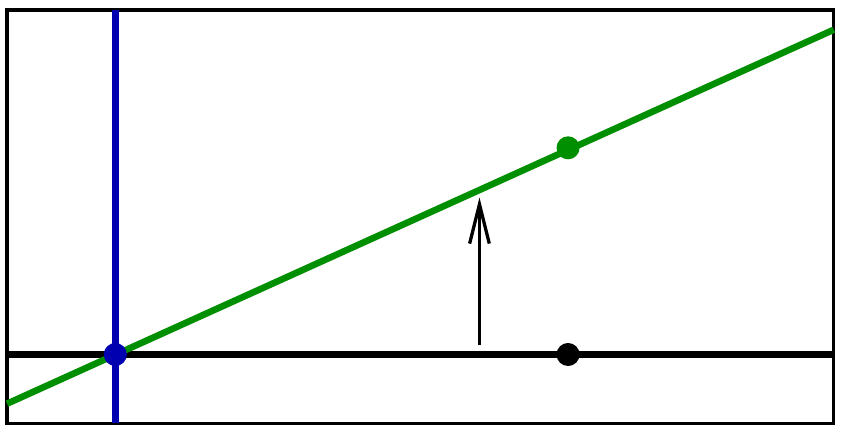}}
% \cell{4.5}{1.5}{c}{\small $0$}
\cell{4}{18}{c}{\small $V$}
\cell{38}{2}{c}{\small $U$}
\cell{33}{18}{c}{\small $W$}
\cell{29}{5}{c}{\small $u$}
\cell{30.5}{12}{c}{\small $i(u)$}
\cell{24}{7}{c}{\small $i$}
\cell{21}{7}{c}{\small $\iso$}
\end{picture}%
\caption{The canonical isomorphism $i$ between complements $U$ and $W$ of $V$}
\label{fig:comps}
\end{figure}

For a linear map $f \from U \to V$ between vector spaces $U$ and $V$, its
graph
\[
W_f = \{ (u, v) \in U \oplus V \such f(u) = v\}
\]
is a complement of $V$ in $U \oplus V$.  Moreover, for any complement $W$
of $V$ in $U \oplus V$, there is a unique linear map $f \from U \to V$ such
that $W_f = W$: in the notation of the previous paragraph, $f(u) = i(u) -
u$.  Hence:
\begin{lemma}
\label{lemma:graphs}
For vector spaces $U$ and $V$, there is a canonical bijection
\[
\{\text{linear maps } U \to V\}
\iso
\{\text{complements of } V \text{ in } U \oplus V\}.
\]
\end{lemma}

Now we turn to nilpotents. Our first lemma is a straightforward induction.

\begin{lemma}
\label{lemma:li}
Let $T$ be a nilpotent operator on a vector space $X$, and let $v \in X$.
Let $k \geq 0$ be least such that $T^k(v) = 0$.  Then $v, T(v), \ldots,
T^{k - 1}(v)$ are linearly independent.
\qed
\end{lemma}
% 
% \begin{proof}
% This is a straightforward induction on $k$.
% \end{proof}

% \begin{proof}
% For $k = 0$, this is trivial. Suppose inductively that it holds for $k$,
% and put $w = T(v)$, so that $k - 1$ is least such that $T^{k - 1}(w) =
% 0$. If $\sum_{i = 0}^{k - 1} c_i T^i(v) = 0$ then applying $T$ to each side
% gives $\sum_{i = 0}^{k - 2} c_i T^i(w) = 0$. By inductive hypothesis, $c_i
% = 0$ for all $i < k - 1$, and then $c_{k - 1} = 0$ too since $T^{k - 1}(v)
% \neq 0$.
% \end{proof}

\begin{lemma}
\label{lemma:block}
Let $U$ and $V$ be vector spaces and let $T$ be a linear operator on $U
\oplus V$ such that $TV \sub V$.  Write $\begin{pmatrix} T_{UU} &0
  \\ T_{UV} & T_{VV} \end{pmatrix}$ for the block decomposition of $T$.
Then
\[
T \text{ is nilpotent} 
\iff
T_{UU} \text{ and } T_{VV} \text{ are nilpotent.}
\]
\end{lemma}

\begin{proof}
For $u \in U$ and $k \geq 0$, the $U$-component of $T^k(u, 0)$ is
$T_{UU}^k(u)$.  Hence if $T^k = 0$ then $T_{UU}^k = 0$, and similarly
$T_{VV}^k = 0$.  Conversely, suppose that $T_{UU}^\ell = 0$ and $T_{VV}^m =
0$.  Then $T^\ell(U \oplus V) \sub V$, so $T^{\ell + m}(U \oplus V) \sub
T^m V = \{0\}$.
\end{proof}

An \demph{automorphism} of a vector space is an invertible operator.  Every
operator decomposes uniquely as the direct sum of a nilpotent and an
automorphism: 

\begin{lemma}[Fitting]
\label{lemma:fitting}
Let $Q$ be a linear operator on a finite-dimensional vector space.  Then
there is a unique pair $(W, V)$ of complementary $Q$-invariant subspaces
such that $Q$ is nilpotent on $W$ and an automorphism of $V$.
\end{lemma}

\begin{proof}
See standard texts such as Jacobson~\cite{JacoBA2}
(Section~3.4). Explicitly, $W = \bigcup_{i \geq 0} \ker(Q^i)$ and $V =
\bigcap_{i \geq 0} \im(Q^i)$.
\end{proof}

Finally, let $X$ be a finite-dimensional vector space.  The set of ordered
bases $(x_1, \ldots, x_n)$ of $X$ is a torsor (in the sense of the
Introduction) over the automorphism group $\Aut(X)$, so there is a
bijection between the set of ordered bases and $\Aut(X)$.

\section{The proof}
\label{sec:pf}

\begin{thm}
\label{thm:main}
Let $X$ be a finite-dimensional vector space over any field. Then there is
a bijection
\[
\Nil(X) \times X \iso \Lin(X),
\]
where $\Nil(X)$ is the set of nilpotent operators on $X$ and $\Lin(X)$ is
the set of all linear operators on $X$.
\end{thm}

\begin{proof}
For each subspace $V$ of $X$, choose a bijection between the ordered bases
of $V$ and the automorphisms of $V$, and choose a complementary subspace
$V^\perp$.  

Let $(T, v) \in \Nil(X) \times X$.  Write $V = \spn\{T^i(v) \such i \geq
0\}$ (Fig.~\ref{fig:lin}\bhardref{a}).
\begin{figure}
\centering%
\lengths%
% \fbox{%
\begin{picture}(120,52)(0,-2)
\cell{0}{3}{bl}{\includegraphics[height=44mm]{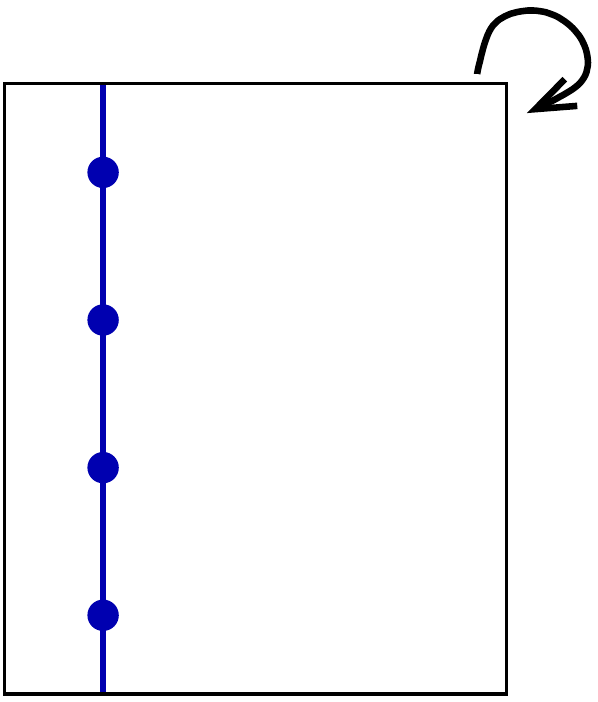}} % width is about 37
\cell{15.5}{0}{c}{(a)}
\cell{30}{45}{r}{nilpotent $T$}
\cell{3}{39.5}{l}{$V$}
\cell{8}{36}{l}{$v$}
\cell{8}{26.7}{l}{$T(v)$}
\cell{8}{17.2}{l}{$T^2(v)$}
\cell{8}{8}{l}{$0$}
\cell{44}{3}{bl}{\includegraphics[height=44mm]{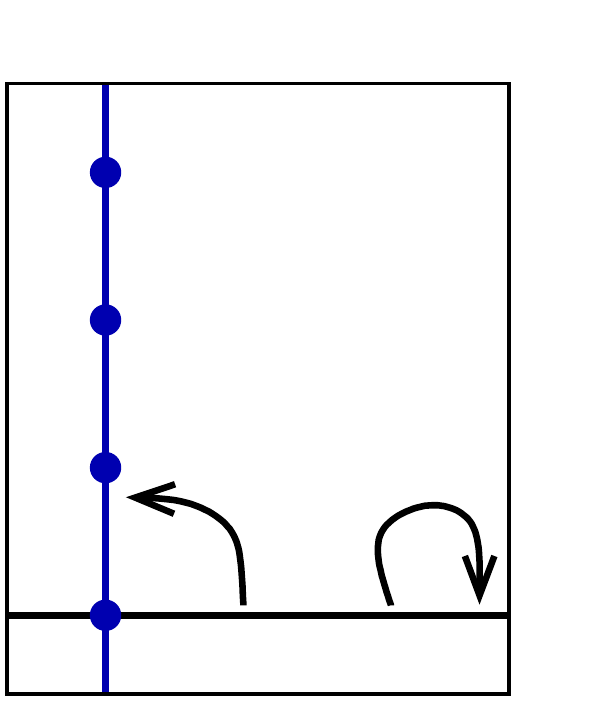}}
\cell{61.5}{0}{c}{(b)}
\cell{47}{39.5}{l}{$V$}
\cell{72}{6}{c}{$V^\perp$}
\cell{49}{36}{r}{$v_0$}
\cell{49}{26.7}{r}{$v_1$}
\cell{49}{17.2}{r}{$v_2$}
\cell{49}{6.5}{r}{$0$}
\cell{60}{17}{c}{$T_{V^{\!\perp}\! V}$}
\cell{71}{21.5}{c}{$T_{V^{\!\perp}\!V^{\!\perp}}$}
\cell{70.5}{18}{c}{nil}
\cell{87.5}{3}{bl}{\includegraphics[height=44mm]{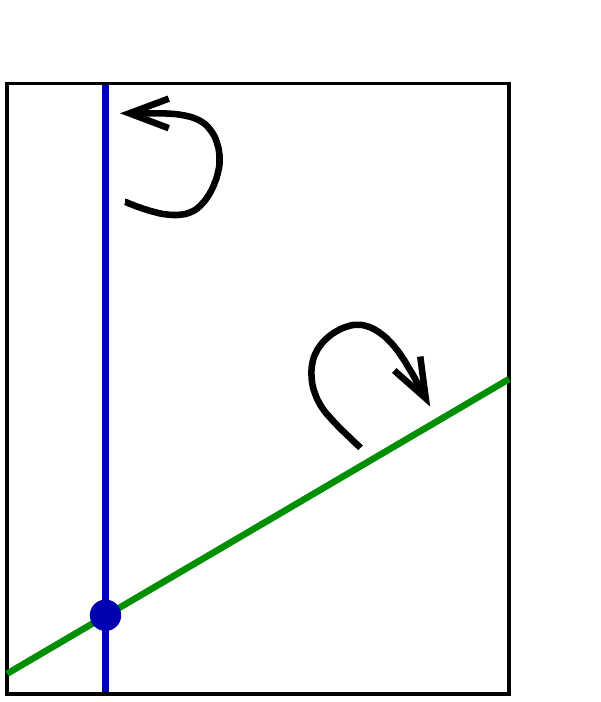}}
\cell{103}{0}{c}{(c)}
\cell{90.5}{37}{l}{$V$}
\cell{115}{17}{c}{$W$}
\cell{106}{24}{r}{$S$}
\cell{106}{20.5}{r}{nil}
\cell{102}{38.5}{l}{$R$}
\cell{102}{35}{l}{auto}
\end{picture}%
% }
\caption{Schematic diagrams of the proof of Theorem~\ref{thm:main}}
\label{fig:lin}
\end{figure}
Let $k \geq 0$ be least such that $T^k(v) = 0$, and put
\[
\vc{v} = (v_0, v_1, \ldots, v_{k - 1}) = (v, T(v), \ldots, T^{k - 1}(v)).
\]
By Lemma~\ref{lemma:li}, $\vc{v}$ is an ordered basis of $V$.  Evidently
$TV \sub V$, and the action of $T$ on $V$ is completely determined by
$\vc{v}$: it is the nilpotent $v_0 \mapsto \cdots \mapsto v_{k - 1} \mapsto
0$.  Also, $v$ is determined by $\vc{v}$, since $v = 0$ if $k = 0$ and $v =
v_0$ otherwise.  The restriction $T|_{V^\perp} \from V^\perp \to X = V
\oplus V^\perp$ decomposes as a linear map $T_{V^\perp V} \from V^\perp \to
V$ and a linear operator $T_{V^\perp V^\perp}$ on $V^\perp$.  Hence by
Lemma~\ref{lemma:block}, to give a pair $(T, v)$ is equivalent to giving a
linear subspace $V$ equipped with an ordered basis, a linear map $V^\perp
\to V$, and a nilpotent on $V^\perp$ (Fig.~\ref{fig:lin}\bhardref{b}).

By Lemma~\ref{lemma:graphs}, the linear maps $V^\perp \to V$ are in
bijection with the subspaces $W$ of $X$ complementary to $V$.  Given such a
$W$, there is a canonical isomorphism $V^\perp \to W$ (as in
Section~\ref{sec:la}), so we can equivalently replace the nilpotent
$T_{V^\perp V^\perp}$ on $V^\perp$ by a nilpotent $S$ on $W$.  Hence to give
a pair $(T, v)$ is equivalent to giving a pair $(V, W)$ of complementary
subspaces together with an ordered basis of $V$ and a nilpotent on $W$.

Now using the bijection chosen at the start of the proof, we can
equivalently replace the ordered basis of $V$ by an automorphism $R$ of $V$
(Fig.~\ref{fig:lin}\bhardref{c}).  Thus, we now have a pair $(V, W)$ of
complementary subspaces of $X$, an automorphism of $V$, and a nilpotent on
$W$.  And by Lemma~\ref{lemma:fitting}, to give such data is exactly to
give a linear operator on $X$.
\end{proof}

It follows that on a vector space with finite cardinality $N$, an operator
chosen uniformly at random has probability $1/N$ of being nilpotent.

% \begin{remarks}
% \begin{enumerate}
% \item 
\paragraph{Remarks}
In both Joyal's proof of Cayley's formula and our proof of
Theorem~\ref{thm:main}, the amount of arbitrary choice can be
reduced.  In the Cayley case, it suffices to choose a single total order
on $X$: for this induces a total order on each subset $V$, hence, via
the torsor argument of the Introduction, a bijection between orders on
and permutations of $V$.  In the linear case, it suffices to choose a
single ordered basis of $X$.  The matrix echelon algorithm then
produces an ordered basis of each subspace $V$, hence (by the torsor
argument) a bijection between ordered bases and automorphisms of $V$.
The Steinitz exchange algorithm produces a complement $V^\perp$
of each subspace $V$.

% \item
The proof of Theorem~\ref{thm:main} establishes slightly more than is
stated.  For a nilpotent $T$ and a vector $v$, let
$\deg_v(T)$ denote the least $k \geq 0$ such that $T^k(v) = 0$. 
Then the proof shows that for all integers $k$, 
\[
\bigl\{ (T, v) \in \Nil(X) \times X \such \deg_v(T) = k\bigr\}
\iso
\biggl\{ 
Q \in \Lin(X) \such 
\dim \biggl( \bigcap_{i \geq 0} \im Q^i \biggr) = k
\biggr\}. 
\]
% \end{enumerate}
% \end{remarks}

\paragraph{Acknowledgements} This work was supported by a Leverhulme Trust
Research Fellowship. I thank Joachim Kock and Darij Grinberg for their
comments.

\bibliography{mathrefs}

\end{document}